\documentclass[12pt,reqno]{amsart}
\usepackage{amssymb,amsmath,amsthm,enumerate,verbatim,bbm}
\usepackage[a4paper]{geometry}

\DeclareMathOperator{\Ran}{Ran}

\DeclareMathOperator{\Ker}{Ker}

\DeclareMathOperator{\Span}{span}

\newcommand{\comp}{\text{\rm comp}}
\newcommand{\ac}{\text{\rm ac}}
\newcommand{\ess}{\text{\rm ess}}

\renewcommand\Re{\hbox{{\rm Re}}\,}
\newcommand{\abs}[1]{\lvert#1\rvert}
\newcommand{\Abs}[1]{\left\lvert#1\right\rvert}
\newcommand{\norm}[1]{\lVert#1\rVert}

\newcommand{\bbR}{{\mathbb R}}
\newcommand{\bbC}{{\mathbb C}}

\newcommand{\bbN}{{\mathbb N}}
\newcommand{\bbZ}{{\mathbb Z}}

\newcommand{\calL}{\mathcal{L}}
\newcommand{\calM}{\mathcal{M}}
\newcommand{\calN}{\mathcal{N}}

\numberwithin{equation}{section}

\theoremstyle{plain}
\newtheorem{theorem}{\bf Theorem}[section]
\newtheorem*{theoremA}{Theorem A}

\newtheorem{lemma}[theorem]{\bf Lemma}
\newtheorem{proposition}[theorem]{\bf Proposition}

\theoremstyle{definition}

\theoremstyle{remark}
\newtheorem*{remark*}{\bf Remark}

\newcommand{\wt}{\widetilde}
\newcommand{\eps}{\varepsilon}
\newcommand{\loc}{\mathrm{loc}}

\renewcommand{\[}{\begin{equation}}

\renewcommand{\]}{\end{equation}}

\begin{document} 
\title{On the spectrum of the multiplicative Hilbert matrix}
\date{\today} 

\author{Karl-Mikael Perfekt}
\address{Department of Mathematics and Statistics, 
  University of Reading, Reading RG6 6AX, United Kingdom}
\email{k.perfekt@reading.ac.uk}

\author{Alexander Pushnitski} 
\address{Department of Mathematics,
King's College London,
Strand, London WC2R 2LS,
United Kingdom}
\email{alexander.pushnitski@kcl.ac.uk}



\begin{abstract}
We study the multiplicative Hilbert matrix, i.e. the infinite matrix with entries
$(\sqrt{mn}\log(mn))^{-1}$ for $m,n\geq2$. 
This matrix was recently introduced within the context of the theory of 
Dirichlet series, and it was shown that the multiplicative Hilbert matrix
has no eigenvalues and that its continuous spectrum coincides with $[0,\pi]$. 
Here we prove that the multiplicative Hilbert matrix has no singular continuous 
spectrum and that its absolutely continuous spectrum has multiplicity one. 
Our argument relies on the tools of spectral perturbation theory and scattering theory.
Finding an explicit diagonalisation of the multiplicative Hilbert matrix remains an interesting
open problem.  
\end{abstract}

\maketitle

\section{Introduction}\label{sec.a}

Let $\{h(n)\}_{n=0}^\infty$ be a sequence of complex numbers. 
A \emph{Hankel matrix} is an infinite matrix of the form 
$$
\{h(n+m)\}_{n,m=0}^\infty,
$$
considered as a linear operator on $\ell^2(\bbZ_+)$, $\bbZ_+=\{0,1,2,\dots\}$. 
One of the key examples of Hankel matrices is the \emph{Hilbert matrix}: 
\[
\{(n+m+a)^{-1}\}_{n,m=0}^\infty,\quad a>0.
\label{a1}
\]
Magnus \cite{Magnus} studied the spectrum of the Hilbert matrix for $a=1$, proving that it is given by the interval $[0,\pi]$ and that it is purely continuous, i.e. there are no eigenvalues.  
Later, Rosenblum \cite{Rosenblum} proved the following theorem.
\begin{theoremA}
For any $a>0$, the absolutely continuous (a.c.) spectrum of the 
Hilbert matrix \eqref{a1} is $[0,\pi]$ with multiplicity one. 
There is no singular continuous spectrum. 
\begin{enumerate}[\rm (i)]
\item
If $a\geq1/2$, there are no eigenvalues. 
\item
If $0<a<1/2$, there is one simple eigenvalue $\pi/\cos(\pi(a-1/2))$, and no other eigenvalues. 
\end{enumerate}
\end{theoremA}
In fact, Rosenblum gave an explicit 
diagonalisation of the Hilbert matrix for all $a\in\bbR$, $a\not=0,-1,-2,\dots$,
from which one can read off the above description of the spectrum.

The Hilbert matrix is remarkable, in particular, for being the simplest bounded
non-compact Hankel matrix.
It exhibits the following borderline behaviour: 
\begin{itemize}
\item
If $nh(n)\to0$ as $n\to\infty$, then $\{h(n+m)\}_{n,m=0}^\infty$ is compact on $\ell^2(\bbZ_+)$;
\item
If $nh(n)\to\infty$ as $n\to\infty$, then $\{h(n+m)\}_{n,m=0}^\infty$ is unbounded on $\ell^2(\bbZ_+)$.
\end{itemize}

In this paper we discuss a similarly remarkable borderline operator in the class of 
\emph{multiplicative Hankel matrices,} i.e. in the class of infinite matrices of the form 
$$
M(g)=\{g(nm)\}_{n,m=1}^\infty,
$$
considered as linear operators on $\ell^2(\bbN)$, $\bbN=\{1,2,\dots\}$.
Here the $(n,m)$'th entry of the matrix depends on the product $nm$ instead of the sum $n+m$.  
Following \cite{QQ}, we call such operators \emph{Helson matrices,}
in honour of H.~Helson's pioneering work \cite{Helson} on the subject. 

The special Helson matrix we consider corresponds to the sequence
\[
g_a(n)=\frac1{\sqrt{n}(a+\log n)}, \quad n\geq1, \quad a>0.
\label{a0}
\]
It is not difficult to see that $M(g_a)$ is a bounded self-adjoint operator
on $\ell^2(\bbN)$ (see Theorem~\ref{thm.a4}). 
Similarly to the case of the Hilbert matrix, the Helson matrix $M(g_a)$ is borderline in the following sense: 
\begin{itemize}
\item
If $(\sqrt{n}\log n) g(n)\to0$ as $n\to\infty$, then $M(g)$ is compact on $\ell^2(\bbN)$; 
\item
If $(\sqrt{n}\log n) g(n)\to\infty$ as $n\to\infty$, then $M(g)$ is unbounded on $\ell^2(\bbN)$.
\end{itemize}

Even though it is not possible to take $a=0$ in the definition \eqref{a0}, 
one can do so if the indices $n,m$ are restricted to 
$\bbN_2=\{2,3,\dots\}$. We will denote by $M_2(g)$ 
the operator on $\ell^2(\bbN_2)$ corresponding to the matrix $\{g(nm)\}_{n,m\geq2}$.
Hence, when $a=0$, one can consider $M_2(g_0)$, with $g_0$ defined as in \eqref{a0}, i.e. 
$$
M_2(g_0)=\{(\sqrt{mn}\log(mn))^{-1}\}_{n,m=2}^\infty. 
$$
Following \cite{BPSSV}, we call $M_2(g_0)$ the \emph{multiplicative Hilbert matrix}. 

We now state our main results.
\begin{theorem}\label{thm.a1}
The multiplicative Hankel matrix $M_2(g_0)$ has a purely a.c. spectrum (no singular continuous spectrum, 
no eigenvalues) which coincides with $[0,\pi]$ and has multiplicity one. 
\end{theorem}

\begin{theorem}\label{thm.a2}
For any $a>0$, the a.c. spectrum of $M(g_a)$ coincides with $[0,\pi]$ and has multiplicity one.
The singular continuous spectrum of $M(g_a)$ is absent. 
There is a critical value $a_*>0$ such that: 
\begin{enumerate}[\rm (i)]
\item
If $a\geq a_*$, then $M(g_a)$ has no eigenvalues. 
\item
If $0< a<a_*$, then $M(g_a)$ has one simple eigenvalue $\lambda(a)>\pi$,  and no other eigenvalues. 
The eigenvalue $\lambda(a)$ is a continuous non-increasing function of $a\in(0,a_*)$, with $\lim_{a \to a_*^{-}} \lambda(a) = \pi$
and $\lim_{a \to 0^{+}} \lambda(a) = \infty$. 
\end{enumerate}
\end{theorem}

Despite the similarity with Theorem~A, we do not have an explicit diagonalisation of either $M(g_a)$
or $M_2(g_0)$; to find one is an interesting open problem. 
We are also unable to explicitly compute the critical value $a_*$, and the same is true for the eigenvalue $\lambda(a)$.
We only have the crude estimates
$$
\frac1\pi\leq a_*\leq 2,\quad \frac1a\leq \lambda(a)\leq \pi+\frac1a,
$$
see Section~\ref{sec.aaa}. 

The multiplicative Hilbert matrix $M_2(g_0)$ was introduced in \cite{BPSSV}, where it was placed in the context of the study of operators acting on 
Hardy spaces of Dirichlet series. The multiplicative Hilbert matrix is especially interesting, because although its analogy to the classical Hilbert matrix is unambiguous, questions about its place within the theory of Dirichlet series are open. In particular, it is not known whether $M_2(g_0)$ has a bounded
symbol --- see Problem~3.2 in \cite{SaksmanSeip} for a precise statement.

The fact that the spectrum of $M_2(g_0)$ is purely continuous and coincides with $[0,\pi]$ 
was already proven in \cite{BPSSV}, but we give a more streamlined version of the proof.

Our main contributions in this paper are:
\begin{itemize}
\item
the proof that the singular continuous spectra of $M_2(g_0)$ and $M(g_a)$ are absent and
that the respective a.c. spectra have multiplicity one; 
\item
a clarification of the connection between $M_2(g_0)$, $M(g_a)$ 
and certain integral Hankel operators. This lies at the heart of our proof --- see Section~\ref{sec.aa}.
\end{itemize}

We also attempt to push the analogy with the Hilbert matrix a little further by introducing the 
one-parameter family $M(g_a)$, with the hope of stimulating some further progress.

\section{The strategy of the proof}\label{sec.aa}

\subsection{Reduction to integral Hankel operators}

Key to our analysis is a unitary equivalence between $M(g_a)$, $M_2(g_0)$ and certain integral Hankel operators. 
We start by recalling the definition of this class of operators. 
For a \emph{kernel function} $h\in L^1_{\loc}(\bbR_+)$, let us denote by $H(h)$ the 
\emph{integral Hankel operator} in $L^2(\bbR_+)$, formally defined by 
$$
(H(h)f)(t)=\int_0^\infty h(t+s)f(s)ds, \quad t>0, \quad f\in L^2(\bbR_+).
$$
We need a simple sufficient condition for the boundedness of integral Hankel operators. 
Let $H(1/t)$ be the Carleman operator, i.e. the integral Hankel operator with the kernel function $h(t)=1/t$. 
Recall that $H(1/t)$ is bounded on $L^2(\bbR_+)$ with norm $\pi$; we will come back to 
the spectral properties of $H(1/t)$ in Section~\ref{sec.b3}. From the boundedness of $H(1/t)$ we immediately
get the following estimate, which is both well known and easy to prove. 
\begin{lemma}\label{lma.d1}
Let $h$ be a kernel function with $\abs{h(t)}\leq A/t$ for some $A>0$.   
Then $\norm{H(h)}\leq A\pi$. 
\end{lemma}
\begin{proof}
For $f_1,f_2\in L^2(\bbR_+)$, we have
$$
\abs{(H(h)f_1,f_2)}\leq A(H(1/t)\abs{f_1},\abs{f_2})\leq A\pi\norm{\abs{f_1}}\norm{\abs{f_2}}=A\pi\norm{f_1}\norm{f_2}.
\qedhere
$$
\end{proof}

Let
$$
\zeta(t)=\sum_{n=1}^\infty n^{-t}, \quad t>1,
$$
be the Riemann zeta function. 
As usual, for a bounded operator $H$ in a Hilbert space, we say that $H$ is non-negative, $H\geq0$, if
$(Hx,x)\geq0$ for all elements $x$ in the Hilbert space. 

Now we are ready to state the unitary equivalence between $M_2(g_0)$, $M(g_a)$, and integral Hankel operators. 
\begin{theorem}\label{thm.a4}
\begin{enumerate}[\rm (i)]
\item
The operator $M_2(g_0)$ is bounded, non-negative, has trivial kernel, and is unitarily equivalent to 
$$
H(h_0)|_{(\Ker H(h_0))^\perp}, \quad h_0(t)=\zeta(t+1)-1.
$$
\item
For any $a>0$, the operator $M(g_a)$ is bounded, non-negative, has trivial kernel, and is unitarily equivalent to 
$$
H(h_a)|_{(\Ker H(h_a))^\perp}, \quad h_a(t)=\zeta(t+1)e^{-at/2}.
$$
\end{enumerate}
\end{theorem}
In fact, it is not difficult to see that the kernels of
$H(h_0)$ and $H(h_a)$ are also trivial. However, we will not need this observation in our proof. 
\begin{proof}
(i) 
Consider the operator
$\calN_0: L^2(\bbR_+)\to \ell^2(\bbN_2)$, given by 
$$
(\calN_0 f)_n=\int_0^\infty n^{-\frac12-t}f(t)dt, 
\quad n\in\bbN_2,\quad f\in L^2(\bbR_+).
$$
For $f_1,f_2\in L^2(\bbR_+)$, 
we have that
\begin{align*}
(\calN_0f_1,\calN_0f_2)
&=
\sum_{n=2}^\infty \int_0^\infty \int_0^\infty n^{-1-t-s}f_1(s)\overline{f_2(t)}\, ds\, dt \\
&=
\int_0^\infty \int_0^\infty (\zeta(t+s+1)-1)f_1(s)\overline{f_2(t)} \, ds\, dt.
\end{align*}
Thus
$$
(\calN_0 f_1,\calN_0f_2)=(H(h_0)f_1,f_2), \quad h_0(t)=\zeta(t+1)-1. 
$$
We have the elementary bound
\[
h_0(t)=\sum_{n=2}^\infty n^{-t-1}\leq\int_1^\infty\frac{dx}{x^{t+1}}=\frac1t, \quad t>0.
\label{d1}
\]
Together with Lemma~\ref{lma.d1}, this bound shows that $H(h_0)$ is bounded
on $L^2(\bbR_+)$. 
Thus, $\calN_0$ is also bounded and 
$$
H(h_0)=\calN_0^*\calN_0.
$$
Next, consider the adjoint $\calN_0^*: \ell^2(\bbN_2)\to L^2(\bbR_+)$, given by
$$
(\calN_0^* x)(t)=\sum_{m\geq2} x_m m^{-\frac12-t}, 
\quad t>0, \quad x=\{x_m\}_{m=2}^\infty\in\ell^2(\bbN_2).
$$
Since
$$
g_0(n)=(\sqrt{n}\log n)^{-1}=\int_0^\infty n^{-\frac12-t}dt,\quad n\geq2,
$$
we see that the operator $\calN_0\calN_0^*$ on $\ell^2(\bbN)$ is given by 
$$
(\calN_0\calN_0^* x)_n=\sum_{m\geq2} g_0(nm)x_m,
\quad x\in \ell^2(\bbN),
$$
i.e. $\calN_0\calN_0^*=M_2(g_0)$. 

It is well known that, for any bounded operator $\calN$, the operators
$$
\calN^*\calN|_{(\Ker\calN)^\perp}
\quad \text{ and }\quad
\calN\calN^*|_{(\Ker\calN^*)^\perp}
$$
are unitarily equivalent. 
This gives the required unitary equivalence between the non-zero parts of $M_2(g_0)$ and $H(h_0)$. 

To complete the proof of part (i), it remains to show that the kernel of $\calN_0^*$ is trivial. 
This is easy to check: if $\calN_0^* x=0$, then, inspecting the asymptotics of $(\calN_0^* x)(t)$ as $t\to\infty$, 
we inductively prove that $x_m=0$ for $m=2,3,\dots$. 

(ii)
For a fixed parameter $a>0$,  let $\calN_a: L^2(\bbR_+)\to \ell^2(\bbN)$ be the linear operator given by 
$$
(\calN_a f)_n=\int_0^\infty e^{-at/2}n^{-\frac12-t}f(t)dt, 
\quad n\in\bbN,\quad f\in L^2(\bbR_+).
$$
Similarly to part (i), we have 
$$
\calN_a^*\calN_a=H(h_a), \quad h_a(t)=\zeta(t+1)e^{-at/2}. 
$$
Using \eqref{d1}, we get
$$
h_a(t)=e^{-at/2}\sum_{n=1}^\infty n^{-t-1}\leq e^{-at/2}(1+1/t)\leq C(a)/t, \quad t > 0,
$$
and so $H(h_a)$ is bounded. 

On the other hand, the adjoint $\calN_a^*: \ell^2(\bbN)\to L^2(\bbR_+)$
is given by 
$$
(\calN_a^* x)(t)=e^{-at/2}\sum_{m\geq1} x_m m^{-\frac12-t}, 
\quad t>0, \quad x=\{x_m\}_{m\geq1}\in\ell^2(\bbN).
$$
Since
$$
g_a(n)=\frac1{\sqrt{n}(a+\log n)}=\int_0^\infty e^{-at}n^{-\frac12-t}dt,
\quad
n\geq1,
$$
we obtain
$$
\calN_a\calN_a^*=M(g_a).
$$
This gives the required unitary equivalence. 
Again it is easy to see that the kernel of $\calN_a^*$ is trivial. 
\end{proof}

\subsection{Some heuristics}
Theorem~\ref{thm.a4} reduces the question to the analysis of integral Hankel operators
with specific kernel functions. Let us recall some basic facts about such operators. 

It is well known that the Carleman operator has a purely a.c. spectrum $[0,\pi]$
of multiplicity \emph{two}. In fact, $H(1/t)$ is explicitly diagonalised by the Mellin transform, 
see Section~\ref{sec.b3} below. 
From Lemma~\ref{lma.d1} it is easy to conclude that 
$$
h(t)=o(1/t) \text{ as } t\to0 \text{ and as } t\to\infty \quad \Rightarrow \quad H(h) \text{ is compact.}
$$
Heuristically, the behaviour $1/t$ of the kernel function is \emph{singular} both as $t\to0$ and as $t\to\infty$. The spectrum of $H(1/t)$ has multiplicity two because each 
of these singularities generates an interval of a.c. spectrum of multiplicity one. 
J.~S.~Howland  has made this observation more precise in \cite{Howland} by proving, 
brushing some technical details aside, that if 
\[
h(t)=\begin{cases} 
c_0/t+\text{error term}, & t\to0,
\\
c_\infty/t+\text{error term}, & t\to\infty,
\end{cases}
\label{a7}
\]
then the a.c. spectrum of $H(h)$ is given by the union of intervals
$$
\sigma_{\ac}(H(h))=[0,\pi c_0]\cup [0,\pi c_\infty],
$$
where each of the two intervals contributes multiplicity one to the spectrum. 
(Howland was motivated, on the one hand, by the earlier work \cite{Power} of S.~Power, which concerns 
the essential spectrum of Hankel operators with piecewise continuous symbols,
and on the other hand, by similar results in scattering theory for Schr\"odinger operators.)
Howland's results were further extended in \cite{PYa}, where a more general class of kernel functions
was considered. 

Now let us come back to the kernel functions $h_a$ and $h_0$ of Theorem~\ref{thm.a4}. 
Recall that zeta function $\zeta(z)$ has a simple pole at $z=1$ with residue $1$ and so 
\[
\zeta(1+t)-1/t\in C^\infty([0,\infty)), 
\label{a11}
\]
and 
\[
\zeta(t)=1+O(2^{-t}), \quad t\to\infty. 
\label{a12}
\]
It follows that $h_a$ satisfies
$$
h_a(t)=
\begin{cases}
1/t+O(1), & t\to0,
\\
o(1/t), & t\to\infty,
\end{cases}
$$
both for $a=0$ and for $a>0$. 
Thus, $h_a$ satisfies \eqref{a7} with $c_0=1$ and $c_\infty=0$ and so, according to Howland's 
paradigm, we should expect $H(h_a)$ to have the a.c. spectrum $[0,\pi]$ of multiplicity one. 
This is indeed what we will prove.

\subsection{The absolutely continuous spectrum}

The statements about the absolutely continuous and the singular continuous spectra of $H(h_a)$ 
are consequences of the following theorem.

\begin{theorem}\label{thm.a4a}
Let $h(t)$, $t>0$, be a real-valued kernel function such that 
\[
\begin{split}
\abs{h(t)-1/t}&\leq Ct^{-1+\eps}, \quad 0<t\leq1,
\\
\abs{h(t)}&\leq Ct^{-1-\eps}, \quad t\geq1,
\end{split}
\label{a8}
\]
with some $\eps>0$ and $C>0$. 
Then the Hankel operator $H(h)$ is bounded and has the essential spectrum $[0,\pi]$. 
The absolutely continuous spectrum of $H(h)$ has multiplicity one and coincides
with the same interval $[0,\pi]$. The singular continuous spectrum of $H(h)$ is absent. 
\end{theorem}

The proof is given in Section~\ref{sec.c}. 
In fact, this theorem follows from \cite[Theorem 2]{Howland}, established with Mourre's inequality, 
or from  \cite[Theorem 7.10]{PYa}, where the smooth method of scattering theory is used.
However, in both of these references the argument is much more complicated than necessary for our purposes. The kernels considered in \cite{Howland, PYa} have two or more singularities, which leads to an a.c. 
spectrum of multiplicity two or more, necessitating the use of multi-channel scattering theory.
Here we give a much simpler single-channel argument. 

Conditions \eqref{a8} can be relaxed somewhat by replacing
$t^{\pm\eps}$ by suitable powers of $\abs{\log t}$, see \cite{PYa}. 

\subsection{The absence of embedded eigenvalues}

The absence of eigenvalues of $H(h_a)$ in the interval $(0,\pi]$ will be established
through the following theorem. 

\begin{theorem}\label{thm.a5}
Let $h$ be as in Theorem~\ref{thm.a4a}. Assume in addition that $H(h)\geq0$,
that the function $\wt h(t)=h(t)-1/t$ satisfies
$\wt h \in C^2([0,\infty))$, and that 
\[
\biggl(\frac{d}{dt}\biggr)^k \wt h(t)=O(t^{-1-k}), \quad t\to\infty, \quad k=1,2.
\label{a9}
\]
Then $H(h)$ has no eigenvalues in $(0,\pi]$.
\end{theorem}

The proof, which is given in Section~\ref{sec.b}, 
is an extension of an argument from \cite{BPSSV}. 
We note that the argument in \cite{BPSSV} was somewhat obscured by the fact that 
the equivalence between $M_2(g_0)$ and the integral Hankel operator $H(h_0)$ was not fully 
understood. Because of this, the argument in \cite{BPSSV} is presented in terms
of functions given by Dirichlet series, which makes it a little more complicated
than necessary. 
Here we give a more streamlined (and more general) version.

The conditions on $h$ in Theorem~\ref{thm.a5} are certainly not optimal, 
but they are sufficient for proving our main results, since
the kernels $h_0$  and $h_a$ satisfy them.

\subsection{The structure of the paper}

In Section~\ref{sec.c} we prove Theorem~\ref{thm.a4a}. 
In Section~\ref{sec.b} we prove Theorem~\ref{thm.a5}. 
In Section~\ref{sec.aaa} we use these two results, and some additional concrete analysis related to the eigenvalues in the interval $(\pi,\infty)$, to prove Theorems~\ref{thm.a1} and \ref{thm.a2}.

\section{The absolutely continuous spectrum: proof of Theorem~\ref{thm.a4a}}\label{sec.c}

\subsection{Preliminaries}

The idea of the proof is as follows.
Consider the integral Hankel operator $H(h_*)$ with the kernel function   
$$
h_*(t)=\frac{e^{-t/2}}{t}, \quad t>0. 
$$
This kernel function satisfies \eqref{a7} with $c_0=1$ and $c_\infty=0$, and so, 
according to Howland's paradigm, we expect the a.c. spectrum of $H(h_*)$ 
to be $[0,\pi]$ with multiplicity one. This is indeed the case, and in fact, 
an explicit diagonalisation of $H(h_*)$ is available 
\cite{Lebedev1,Lebedev2,Rosenblum,Yafaev1}. This operator has 
a purely a.c. spectrum $[0,\pi]$ 
with multiplicity one (no singular continuous spectrum, no eigenvalues), and the family of generalised eigenfunctions
of the continuous spectrum is known in explicit form, as described in the next subsection. 

Let $h$ be as in Theorem~\ref{thm.a4a}; 
as a warm-up, let us check that $\sigma_\ess(H(h))=[0,\pi]$. 
We have
$$
H(h)=H(h_*)+H(w),
$$
where $w(t)=h(t)-h_*(t)$ satisfies
$$
w(t)=\begin{cases}
O(t^{-1+\eps}),& t\to0,
\\
O(t^{-1-\eps}),&t\to\infty.
\end{cases}
$$
It follows that $\int_0^\infty t\abs{w(t)}^2dt<\infty$, and therefore the integral 
kernel $w(t+s)$ of $H(w)$ is in $L^2(\bbR_+\times\bbR_+)$. 
Thus, $H(w)$ is Hilbert-Schmidt, and therefore compact. 
By Weyl's theorem on the invariance of the essential spectrum under compact 
perturbations, we obtain that $\sigma_\ess(H(h))=\sigma_\ess(H(h_*))=[0,\pi]$, as required. 

Now let us outline the idea of proof of the rest of Theorem~\ref{thm.a4a}. 
Under our assumptions, we will prove that  $H(h)$  has the representation
\begin{equation}
H(h)=H(h_*)+G^*AG,
\label{c0}
\end{equation}
where $G$ is a strongly $H(h_*)$-smooth operator, in the terminology of \cite{Yafaev2},
and $A$ is a compact operator. 
Roughly speaking, this means that the difference $H(h)-H(h_*)$ can be represented 
as an operator with a  sufficiently 
regular integral kernel \emph{in the spectral representation of $H(h_*)$}. 
The proof that \eqref{c0} holds consists of two ingredients: a detailed analysis of the explicit
diagonalisation of $H(h_*)$ together with an identification of a class of $H(h_*)$-smooth operators, and
an (easy) proof of the compactness of the operator $A$.

By standard results of smooth scattering theory,
the representation \eqref{c0} implies the existence and completeness of wave operators 
for the pair of operators $(H(h),H(h_*))$, yielding the statement about the 
a.c. spectrum of $H(h)$. 
The same considerations of scattering theory also yield 
the absence of the singular continuous spectrum of $H(h)$.

\subsection{Diagonalization of $H(h_*)$}
Let $K_\nu(z)$ be the modified Bessel function of the third kind; 
for $\Re \nu>-1/2$ and $\Re z>0$ it is given by the integral 
representation \cite[Section 7.3.5, formula (15)]{BE}
\[
\Gamma(\nu+\tfrac12)K_\nu(z)
=
\sqrt{\pi}\bigl(\tfrac{z}{2}\bigr)^\nu
\int_1^\infty e^{-zu}(u^2-1)^{-\frac12+\nu}du. 
\label{c1}
\]
For $k>0$ and $t>0$, set
$$
\psi_k(t)=\frac1\pi \sqrt{2k\sinh(\pi k)}t^{-\frac12}K_{ik}(t/2).
$$
Formally, $\{\psi_k\}_{k> 0}$ gives a complete normalised set 
of generalised eigenfunctions of $H(h_*)$: 
$$
H(h_*)\psi_k=\lambda(k)\psi_k, \quad\text{ where }\quad \lambda(k)=\frac{\pi}{\cosh(\pi k)}, \quad k>0.
$$
More precisely, we have the following statement  \cite{Lebedev1,Lebedev2,Rosenblum,Yafaev1}.
\begin{proposition}
For $f\in C_\comp^\infty(\bbR_+)$, let 
$$
(Uf)(k)=\int_0^\infty f(t)\psi_k(t)dt, \quad k>0.
$$
Then $U$ extends to a unitary operator on $L^2(\bbR_+)$. 
For any $f\in L^2(\bbR_+)$, it holds that
\[
(UH(h_*)U^* f)(k)=\lambda(k)f(k), \quad k>0.
\label{c2}
\]
\end{proposition}
 One reads off the spectrum of $H(h_*)$ from \eqref{c2}:
it is given by the closure of the range of $\lambda(k)$, $k>0$,  
which coincides with the interval $[0,\pi]$. 
Additionally, since $\lambda(k)$ is a continuous and strictly decreasing function of $k \geq 0$, $\lim_{k \to \infty} \lambda(k) = 0$, the spectrum is purely a.c. and has multiplicity one. More explicitly, the unitary operator $U_0:L^2(\bbR_+)\to L^2(0,\pi)$,
$$
(U_0f)(\lambda)=\bigl(\tfrac{d\lambda(k)}{dk}\bigr)^{-1/2}(Uf)(k), \quad \lambda=\lambda(k)\in(0,\pi),
$$
reduces $H(h_*)$ to the operator of multiplication by
the independent variable on $L^2(0,\pi)$,
$$
(U_0 H(h_*)U_0^*f)(\lambda)=\lambda f(\lambda), 
\quad \lambda\in(0,\pi).
$$

\subsection{Strong smoothness}
Let us fix the following function $q\in L^\infty(\bbR_+)$,
$$
q(t)=
\begin{cases}
\abs{\log t}^{-1}, & t\in(0,1/2),
\\
(\log 2)^{-1}, & t\geq 1/2.
\end{cases}
$$
Denote by $Q$ the operator of multiplication by $q(t)$ in $L^2(\bbR_+)$.
For $\beta>0$, we will also consider the power $Q^\beta$. 
In this subsection we will prove the following result. 
\begin{theorem}\label{thm.b2}
Let $\beta>1/2$. Then, in the terminology 
of \cite[Section 4.4]{Yafaev2}, $Q^\beta$ is strongly $H(h_*)$-smooth 
with any exponent $\gamma$ in the range $0<\gamma<\min\{1,\beta-1/2\}$
on any compact sub-interval of $(0,\pi)$.
\end{theorem}

The strong smoothness here means the following. Let $\delta\subset(0,\pi)$ 
be a compact interval and let $\gamma$ be as in the theorem. Then there exists
a constant $C=C(\delta,\gamma)$ such that for any $f\in L^2(\bbR_+)$
we have the estimates
\begin{gather*}
\abs{(U_0Q^\beta f)(\lambda)}\leq C\norm{f}, \quad \lambda\in\delta,
\\
\abs{(U_0Q^\beta f)(\lambda)-(U_0Q^\beta f)(\lambda')}\leq C\abs{\lambda-\lambda'}^\gamma\norm{f}, 
\quad \lambda,\lambda'\in\delta.
\end{gather*}
In other words, the linear functional 
$$
L^2(\bbR_+)\ni f\mapsto (U_0Q^\beta f)(\lambda)
$$
is norm-H\"older continuous in $\lambda\in\delta$ with the exponent $\gamma$. 

\begin{proof}[Proof of Theorem~\ref{thm.b2}]
Since $\lambda=\lambda(k)$ is a $C^\infty$-smooth function of $k>0$ 
and the derivative $\lambda'(k)$ does not vanish for $k>0$, it suffices to prove that, 
for any compact interval $\Delta\subset (0,\infty)$ and for any $\gamma<\min\{1,\beta-1/2\}$, 
the linear functional 
$$
L^2(\bbR_+)\ni f\mapsto (UQ^\beta f)(k)
$$
is norm-H\"older continuous in $k\in\Delta$ with the exponent $\gamma$. 
Recalling the formula for $U$, we see that this functional is given by 
$$
f\mapsto \int_0^\infty f(t)  \psi_k(t) q(t)^\beta dt, \quad k>0.
$$
Thus, we need  to prove the estimates
\begin{gather*}
\int_0^\infty \abs{\psi_k(t)}^2 q(t)^{2\beta} dt\leq C, \quad k\in\Delta,
\\
\int_0^\infty \abs{\psi_k(t)-\psi_{k'}(t)}^2 q(t)^{2\beta} dt\leq C\abs{k-k'}^{2\gamma}, \quad k,k'\in\Delta.
\end{gather*}
By the explicit form of $\psi_k$, it suffices to prove the estimates 
\begin{gather}
\int_0^\infty \abs{K_{ik}(t/2)}^2 t^{-1}q(t)^{2\beta} dt\leq C, \quad k\in\Delta,
\label{c10}
\\
\int_0^\infty \abs{K_{ik}(t/2)-K_{ik'}(t/2)}^2 t^{-1} q(t)^{2\beta} dt\leq C\abs{k-k'}^{2\gamma}, \quad k,k'\in\Delta.
\label{c11}
\end{gather}
Let us split the domain of integration in \eqref{c10} and \eqref{c11} into $(0,1/2)$ 
and $(1/2,\infty)$ and estimate the corresponding integrals separately. 

Consider first the integrals over $(0,1/2)$. We recall the following representation 
for the modified Bessel function $K_\nu$: 
$$
K_\nu(z)=\frac{\pi}{2\sin(\nu\pi)}(I_{-\nu}(z)-I_\nu(z)), 
$$
where $I_\nu(z)$ is the modified Bessel function of the first kind, given by the convergent series
$$
I_\nu(z)
=
\sum_{m=0}^\infty \frac{(z/2)^{2m+\nu}}{m!\Gamma(m+\nu+1)}. 
$$
For $k>0$, let us write 
$$
I_{\pm ik}(z)=(z/2)^{\pm ik}\wt I_{\pm ik}(z), \quad
\wt I_{\pm ik}(z)=\sum_{m=0}^\infty \frac{(z/2)^{2m}}{m!\Gamma(m\pm ik+1)}. 
$$
By inspection, both $\wt I_{\pm ik}(z)$ and $(d/dk)\wt I_{\pm ik}(z)$
are entire functions of $z$, bounded uniformly for $z\in(0,1)$ and $k\in \Delta$. 
Using the elementary estimate $\abs{e^{ia}-e^{ib}}\leq 2\abs{a-b}^\gamma$ for  $0<\gamma<1$,
we get 
$$
\abs{(z/2)^{ik}-(z/2)^{ik'}}
=
\abs{e^{ik\log(z/2)}-e^{ik'\log(z/2)}}
\leq 
2 \abs{k-k'}^\gamma\abs{\log (z/2)}^\gamma. 
$$
Using this, 
we obtain the estimates
\begin{align*}
\abs{K_{ik}(t/2)}&\leq C, 
\quad k\in \Delta, \quad t\in (0,1/2),
\\
\abs{K_{ik}(t/2)-K_{ik'}(t/2)}&\leq
C\abs{k-k'}^\gamma \abs{\log t}^\gamma, 
\quad k,k'\in \Delta, \quad t\in (0,1/2).
\end{align*}
Now it is clear that for $\beta>\gamma+1/2$ the estimates
\eqref{c10} and \eqref{c11} hold with the integrals taken over $(0,1/2)$. 

Next, let us consider the integrals over $(1/2,\infty)$. 
Here we use the integral representation \eqref{c1} for $K_\nu$.
Let us rewrite it as follows,
$$
K_{ik}(t/2)
=\frac{\sqrt{\pi}}{\Gamma(ik+\tfrac12)}(t/4)^{ik}
\int_1^\infty e^{-tu/2}(u^2-1)^{-\frac12+ik}du. 
$$
For $t\geq1/2$, one can estimate the exponential in the integral for $K_{ik}$ as
$$
e^{-tu/2}=e^{-tu/4}e^{-tu/4}\leq e^{-t/4} e^{-u/8}.
$$
This allows one to conclude that $K_{ik}$ satisfies
$$
\abs{K_{ik}(t/2)}+\abs{(d/dk) K_{ik}(t/2)}\leq C \abs{\log (t/4)}e^{-t/4}, \quad t>1/2, \quad k\in \Delta.
$$
Thus, we obtain the estimates \eqref{c10} and \eqref{c11} with the integrals taken 
over $(1/2,\infty)$. Here we do not need any restrictions on $\gamma$ and $\beta$. 
\end{proof}

\subsection{Putting the results together}

We use the following statement from scattering theory. 
For the proof and further details, see \cite[Section~4.7]{Yafaev2}.

\begin{proposition}\label{prp.c3}
Let $H_0$ be a bounded self-adjoint operator in a Hilbert space. 
Assume that the spectrum of $H_0$ is purely a.c., has constant 
multiplicity $m$, and coincides with the interval $[a,b]$. 
Let $G$ be a bounded operator, which is strongly $H_0$-smooth
with an exponent $\gamma>1/2$ on any compact sub-interval of $(a,b)$. 
Let 
$$
H=H_0+G^* AG, 
$$
where $A$ is a compact self-adjoint operator. Then
\begin{enumerate}[\rm (i)]
\item
the a.c. spectrum of $H$ coincides with $[a,b]$ and has a constant
multiplicity $m$,
\item
$H$ has no singular continuous spectrum,
\item
and all eigenvalues of $H$, distinct from $a$ and $b$, have finite 
multiplicities and can accumulate only to $a$ and $b$. 
\end{enumerate}
\end{proposition}

We will not need part (iii) of this proposition, since the absence of eigenvalues
is proven in Section~\ref{sec.b} by a different method.

\begin{proof}[Proof of Theorem~\ref{thm.a4a}]
We fix $\beta>1$ and write 
$$
H(h)=H(h_*)+Q^\beta AQ^\beta,
$$
where $A$ is the integral operator with the integral kernel 
$$
a(t,s)=q(t)^{-\beta}(h(t+s)-h_*(t+s))q(s)^{-\beta}, \quad t,s\in\bbR_+.
$$
In view of  \eqref{a8} it is easy to see that $a\in L^2(\bbR_+\times\bbR_+)$ and 
so $A$ is Hilbert-Schmidt, hence compact. 
Now it remains to use  Proposition~\ref{prp.c3} with $H=H(h)$, $H_0=H(h_*)$ and $G=Q^\beta$. 
\end{proof}

\section{Absence of embedded eigenvalues: proof of Theorem~\ref{thm.a5}}\label{sec.b}

\subsection{Preliminaries}

The key element of the proof  is the following lemma.

\begin{lemma}\label{lma.b3}
Let $H(h)$ be as in Theorem~\ref{thm.a5}. 
If $H(h)f=Ef$ for some function $f\in L^2(\bbR_+)$ and a constant $E$, $0<E\leq\pi$, then 
$f'\in L^2(\bbR_+)$ and $f(0)=0$. 
\end{lemma}

Before embarking on the proof of Lemma~\ref{lma.b3}, let us show how it leads to a proof of
Theorem~\ref{thm.a5}. 

\begin{proof}[Proof of Theorem~\ref{thm.a5}]
Differentiating the eigenvalue equation $H(h)f=Ef$, we get
\[
\int_0^\infty h'(t+s)f(s)ds=Ef'(t), \quad t>0.
\label{b1a}
\]
Integrating by parts, we obtain 
$$
H(h)f'=-Ef';
$$
the boundary term at zero vanishes because $f(0)=0$ by Lemma~\ref{lma.b3}.
This means that $-E$ is an eigenvalue of $H(h)$ with the eigenfunction $f'$, 
and $f'\in L^2(\bbR_+)$ by Lemma~\ref{lma.b3}. 
Since  $H(h)\geq0$, 
this is impossible unless $f'\equiv0$, which implies that $f\equiv0$. 
\end{proof}

\subsection{Lemmas on the Mellin transform}\label{sec.b3}

The remainder of this section is devoted to proving Lemma~\ref{lma.b3}. 
Our main tool, following \cite{BPSSV}, is the Mellin transform. 
See for example \cite{Titchmarsh} for the background. For $f\in L^2(\bbR_+)$, the 
Mellin transform is defined by 
$$
\calM f(z)=\int_0^\infty s^{z-1}f(s)ds,  \quad z\in\bbC,
$$
as long as the integral converges. 
It can be shown that the Mellin transform, initially defined on a suitable dense subset
of $L^2(\bbR_+)$, extends to an isometry between $L^2(\bbR_+)$ and
the $L^2$ space on the vertical line $\Re z=1/2$. In other words, the Plancherel identity 
\[
\frac1{2\pi}\int_{-\infty}^\infty \abs{\calM f(\tfrac12+i\tau)}^2 d\tau
=
\int_0^\infty \abs{f(s)}^2ds, \quad f\in L^2(\bbR_+)
\label{b2}
\]
holds. 
In this context, the inversion formula for the Mellin transform reads as
\[
f(s)=\frac1{2\pi i}\int_{\frac12-i\infty}^{\frac12+i\infty}s^{-z}\calM f(z)dz, \quad f\in L^2(\bbR_+).
\label{b3}
\]
The Mellin transform is useful to us because it diagonalises the 
Carleman operator $H(1/t)$. More precisely, we have the identity
\[
\calM H(1/t)f(z)=\frac{\pi}{\sin(\pi z)}\calM f(z), \quad \Re z=1/2,\quad f\in L^2(\bbR_+),
\label{b3a}
\]
which is the consequence of the elementary formula
$$
\int_0^\infty \frac{s^{z-1}}{t+s}ds=\frac{\pi}{\sin(\pi z)}t^{z-1}, \quad 0<\Re z<1.
$$

\begin{lemma}\label{lma.b4}
Let $h$, $\wt h$ be as in Theorem~\ref{thm.a5}, and let $g=H(\wt h)f$ for some 
$f\in \Ran H(h)$. 
Then the Mellin transform $\calM g(z)$ extends to a meromorphic
function in the strip 
\[
-\frac12-\eps<\Re z<\frac12+\eps
\label{b4}
\]
where $0 < \eps < 1/2$ is as in \eqref{a8}.
This meromorphic extension has at most one pole in the strip; this pole is simple and is located
at the origin.
The function $\calM g$ satisfies the estimate
\[
\int_{-\infty}^\infty \abs{ (\sigma+i\tau)\calM g(\sigma+i\tau)}d\tau\leq C(\eps'),
\quad
-\frac12-\eps'<\sigma <\frac12+\eps',
\label{b5}
\]
with any $0 < \eps' < \eps$. 
\end{lemma}
\begin{proof}
Since  $f\in \Ran H(h)$, from \eqref{a8} and Cauchy-Schwarz we obtain 
\[
f(t)=\begin{cases}
O(t^{-\frac12}), & t\to0,
\\
O(t^{-\frac12-\eps}), & t\to\infty.
\end{cases}
\label{b11}
\]
Next, we have
$$
g^{(k)}(t)=\int_0^\infty \wt h^{(k)}(t+s)f(s)ds,\quad k=0,1,2.
$$
Combining \eqref{b11} with our assumptions on $\wt h$, we obtain that $g\in C^2([0,\infty))$ and 
\[
g^{(k)}(t)=O(t^{-k-\frac12-\eps}), \quad t\to\infty.
\label{b12}
\]

Further, 
for $\Re z=1/2$, integrating by parts twice, we get
\[
z\calM g(z)
=
-\int_0^\infty s^z g'(s)ds
=
\frac1{z+1}\int_0^\infty s^{z+1}g''(s)ds. 
\label{b13}
\]
The integrals here converge absolutely by the estimates \eqref{b12}; 
the boundary terms vanish by the same estimates. 
The same estimates also show that the right side in \eqref{b13} 
has an analytic extension into the strip \eqref{b4}.

Finally, again by \eqref{b12} we have
$$
\int_0^\infty (\abs{g''(s)}^2+\abs{s^{2+\eps'}g''(s)}^2)ds\leq C(\eps'), \quad 0<\eps'<\eps.
$$
By the Plancherel identity \eqref{b2} for the Mellin transform applied to $s^\alpha g''(s)$ with $0\leq \alpha\leq 2+\eps'$, 
we obtain that 
$\int_0^\infty s^{z+1}g''(s)ds$ is in $L^2$ on the vertical lines $\sigma+i\bbR$ 
with $-1/2-\eps'\leq\sigma\leq 1/2+\eps'$. 
Taking into account the factor $1/(z+1)$ in front of the integral in \eqref{b13}, by
Cauchy-Schwarz we 
arrive at the required bound \eqref{b5}.
\end{proof}

\begin{lemma}\label{lma.b5}
Let $h$ be as in Theorem~\ref{thm.a5}, and let $H(h)f=Ef$ for some 
$f\in L^2(\bbR_+)$ and some $0<E\leq\pi$.  
Then the Mellin transform $\calM f(z)$ extends to an analytic
function in the strip \eqref{b4}, satisfying the estimate 
\[
\int_{-\infty}^\infty \abs{ (\sigma+i\tau)\calM f(\sigma+i\tau)}d\tau\leq C(\eps'),
\quad
-\frac12-\eps'<\sigma<\frac12+\eps',
\label{b9}
\]
for every $0 < \eps' < \eps$.
\end{lemma}
\begin{proof}
Let us write the eigenvalue equation for $f$ as
$$
H(1/t)f-Ef=-H(\wt h)f.
$$
Applying the Mellin transform, letting $g=H(\wt h)f$,  and using \eqref{b3a}, we obtain the equation
\[
\calM f(z)=-\calM g(z)/u(z), \quad u(z)=\frac{\pi}{\sin(\pi z)}-E.
\label{b10}
\]
Initially, this formula is valid for $\Re z=1/2$, but Lemma~\ref{lma.b4} 
ensures that the right side has a meromorphic extension to the strip \eqref{b4}. 

Consider the poles of this extension. By Lemma~\ref{lma.b4}, there may be a pole 
at $z=0$; however, this pole is cancelled out by the pole of $u(z)$ at $z=0$. 
There may also be poles arising due to the zeros of $u(z)$. 
Inspecting these, we find that the only zeros of $u(z)$ in the strip \eqref{b4} are given by 
$$
z=\frac12\pm i\theta, \quad
\theta=\frac1\pi\log\left((\pi/E)-\sqrt{(\pi/E)^2-1}\right).
$$
If $E=\pi$, these two zeros coalesce into one double zero at $z=1/2$. 
However, by the Plancherel identity \eqref{b2} for the Mellin transform, 
the integral of $\abs{\calM f(z)}^2$ over the vertical line $\Re z=1/2$
must be finite. This shows that $\calM f(z)$ in fact cannot have poles on this line. 

Summarizing, we see that \eqref{b10} defines an \emph{analytic} extension of $\calM f$ 
into the strip \eqref{b4}. The estimate \eqref{b9} follows from the estimate \eqref{b5} 
and from the fact that $u(\sigma+i\tau)\to -E$ as $\abs{\tau}\to\infty$. 
\end{proof}

\subsection{Proof of Lemma~\ref{lma.b3}}
By  \eqref{a9}, we have
$$
h'(t)=O(t^{-2}),\quad t\to\infty.
$$
From here, applying the Cauchy-Schwarz inequality to \eqref{b1a}, we obtain 
$$
f'(t)=O(t^{-3/2}),\quad t\to\infty.
$$
Thus, it remains to inspect the behaviour of $f(t)$ and $f'(t)$ for small $t$. 

By the analyticity of $\calM f(z)$ in the strip \eqref{b4} and by the estimate \eqref{b9},
we can shift the contour
of integration in the Mellin inversion formula \eqref{b3} to $\Re z=\sigma$ for any $\sigma$ satisfying $-\frac12-\eps<\sigma<\frac12$. This gives us that
$$
f(s)=\frac1{2\pi i}\int_{\sigma-i\infty}^{\sigma+i\infty} s^{-z}\calM f(z) dz, 
\quad
f'(s)=-\frac1{2\pi i}\int_{\sigma-i\infty}^{\sigma+i\infty}zs^{-z-1}\calM f(z) dz.
$$
Again using the estimate \eqref{b9}, we see that 
$$
\abs{f(s)}\leq Cs^{-\sigma}, \quad \abs{f'(s)}\leq Cs^{-\sigma-1}, \quad s>0.
$$
It follows that $f(0)=0$ and that $f'\in L^2(0,1)$. 
\qed

\section{Proving Theorems~\ref{thm.a1} and \ref{thm.a2}}\label{sec.aaa}

\begin{proof}[Proof of Theorem~\ref{thm.a1}]
By Theorem~\ref{thm.a4}(i), it suffices to check that $H(h_0)|_{(\Ker H(h_0))^\perp}$ has a purely a.c. spectrum $[0,\pi]$ of multiplicity one. 
By the properties of the zeta function (see \eqref{a11}, \eqref{a12}) the function $h_0(t)-1/t$ is analytic near $t=0$ and satisfies
$$
h_0^{(k)}(t)=O(2^{-t}), \quad t\to\infty, \quad k=0,1,2. 
$$
Thus, $h_0$ satisfies the hypotheses of both Theorem~\ref{thm.a4a} and Theorem~\ref{thm.a5}. 
It follows that $H(h_0)$ has the a.c. spectrum $[0,\pi]$ of multiplicity one, no singular continuous spectrum, 
and no eigenvalues in $(0,\pi]$. It remains to rule out the eigenvalues in $(\pi,\infty)$. 
But by the estimate \eqref{d1} and by Lemma~\ref{lma.d1}, we have $\norm{H(h_0)}\leq\pi$, and so 
$H(h_0)$ has no eigenvalues in $(\pi,\infty)$. 
\end{proof}

\begin{proof}[Proof of Theorem~\ref{thm.a2}]

1) 
Theorem~\ref{thm.a4}(ii) reduces the question to analysing the spectrum of 
$H(h_a)|_{(\Ker H(h_a))^\perp}$. 
As in the proof of Theorem~\ref{thm.a1}, by the properties of the zeta function 
it is straightforward to check that the kernel function $h_a$ 
satisfies the hypotheses of both  Theorem~\ref{thm.a4a} and Theorem~\ref{thm.a5}. 
It follows that $H(h_a)$ has the a.c. spectrum $[0,\pi]$ of multiplicity one, no singular continuous spectrum, 
and no eigenvalues in $(0,\pi]$.
So it remains to analyse the eigenvalues of $M(g_a)$ (which coincide with those of $H(h_a)$) 
in the interval $(\pi,\infty)$.

2) For $E\geq\pi$, let us denote by $N((E,\infty);M(g_a))$ the total number of eigenvalues 
of $M(g_a)$ in the interval $(E,\infty)$, counting multiplicity. 
Let us prove that $N((E,\infty);M(g_a))$ is non-increasing in $a>0$. 
It suffices to prove that $M(g_a)$ is monotone non-increasing in the standard quadratic form 
sense, i.e. 
\[
(M(g_a)x,x)\leq (M(g_b)x,x), \quad 0<b<a, \quad x=\{x_n\}_{n=1}^\infty\in \ell^2(\bbN).
\label{a5a}
\]
By the calculation in the proof of  Theorem~\ref{thm.a4}, we have that
$$
(M(g_a)x,x)=(\calN_a^* x,\calN_a^* x)=\int_0^\infty e^{-at}\Abs{f(t)}^2dt, 
\quad f(t)=\sum_{n=1}^\infty x_n n^{-\frac12-t}. 
$$
The required monotonicity \eqref{a5a} obviously follows from this representation. 
 
3)
Let us prove that $M(g_a)$ is continuous in $a>0$ in the operator norm. 
Taking $0<b<a$, we have, as in the previous step, and with $f(t)$ as above,
\begin{multline*}
((M(g_b)-M(g_a))x,x)=\int_0^\infty (e^{-bt}-e^{-at})\abs{f(t)}^2 dt
\\
=\int_0^\infty (1-e^{-(a-b)t})e^{-bt}\abs{f(t)}^2 dt
\leq
\sup_{t>0}e^{-bt/2}(1-e^{-(a-b)t})\norm{M(g_{b/2})x}^2.
\end{multline*}
This supremum tends to zero as $a\to b$, and we therefore obtain the desired claim.

4) 
By \eqref{d1}, we have
$$
h_a(t)=e^{-at/2}\sum_{n=1}^\infty n^{-t-1}
\leq
e^{-at/2}(1+1/t)=\frac{e^{-at/2}}{t}(1+t)\leq \frac{e^{-at/2}e^t}{t},
$$
and so for $a\geq2$ we have $h_a(t)\leq 1/t$. Thus, by Lemma~\ref{lma.d1}, 
$H(h_a)$ has norm less or equal to $\pi$ for $a\geq2$, and hence the same is true for $M(g_a)$.
Thus, for $a\geq2$ the operator $M(g_a)$ has no eigenvalues in $(\pi,\infty)$. 

5) 
Let us check that for any $a>0$, the operator $H(h_a)$ has at most one eigenvalue in $(\pi,\infty)$. 
Let
$$
h_a(t)=e^{-at/2}+w_a(t), \quad 
w_a(t)=e^{-at/2}\sum_{n=2}^\infty n^{-t-1}, \quad t>0.
$$
By \eqref{d1}, we have $w_a(t)\leq 1/t$, 
and therefore,  by Lemma~\ref{lma.d1}, $H(w_a)$ has no eigenvalues in $(\pi,\infty)$.  
On the other hand, the Hankel operator corresponding to the kernel function $e^{-at/2}$
is the rank one operator with the integral kernel $e^{-a(t+s)/2}$, which we naturally denote 
by $(\cdot,e^{-at/2})e^{-at/2}$. 
It has the single non-zero eigenvalue $1/a$. Note that
$$
H(h_a)=(\cdot,e^{-at/2})e^{-at/2}+H(w_a). 
$$
From here our claim follows by a standard argument in perturbation theory.
Indeed, for $E\geq\pi$, let us write the min-max principle in the form
\[
N((E,\infty);H(h_a))
=
\sup\{\dim L: (H(h_a)f,f)>E\norm{f}^2 \quad \forall f\in L \setminus \{0\}\},
\label{a6a}
\]
where the supremum is taken over all subspaces $L$ with the indicated property. 
We claim that 
\[
N((\pi,\infty);H(h_a))\leq1.
\label{a6b}
\] 
Assume that this  is false and take a subspace
$L$ as in \eqref{a6a} with $\dim L\geq2$. 
Then there is a non-zero element $f\in L\cap\{e^{-at/2}\}^\perp$ which satisfies
$$
(H(h_a)f,f)=(H(w_a)f,f)\leq\pi\norm{f}^2, 
$$
in contradiction with the inequality in \eqref{a6a}.

6) 
Next, let us check for $0<a<1/\pi$ that the operator $H(h_a)$ has 
at least one eigenvalue in $(\pi,\infty)$.  
We claim that 
\[
 N((\pi,\infty);H(h_a))\geq1, \quad a<1/\pi.
\label{a6c}
\] 
To prove  \eqref{a6c}, 
let $f=e^{-at/2}$. 
Observe that $H(w_a)\geq0$; this follows by the same argument as in the proof of Theorem~\ref{thm.a4}.
Thus,
$$
(H(h_a)f,f) \geq (f,e^{-at/2})(e^{-at/2},f)=1/a\norm{f}^2>\pi\norm{f}^2,
$$
and so \eqref{a6c} holds by choosing $L=\Span\{e^{-at/2}\}$ in the min-max principle. 

7)
Let us put together the above steps. 
By \eqref{a6b} and \eqref{a6c}, $N((\pi,\infty);M(g_a))=1$ for $0<a<1/\pi$,
and, by step 4),
$N((\pi,\infty);M(g_a))=0$ for $a\geq2$. By the monotonicity in step 2), there must exist
a critical $a_*>0$ such that 
$$
N((\pi,\infty); M(g_a))=
\begin{cases}
0 & a>a_*,
\\
1 & a<a_*. 
\end{cases}
$$
The norm continuity of $M(g_a)$ ensures that $N((\pi,\infty); M(g_a))$ is lower semi-continuous in $a$, and so 
$N((\pi,\infty); M(g_{a_*}))=0$. 

Thus, we have exactly one eigenvalue $\lambda(a)$ for $0<a<a_*$ and no eigenvalues for $a\geq a_*$. 
The norm continuity and monotonicity of $M(g_a)$ ensures that $\lambda(a)$ is a continuous monotone function of $a$ with
$\lambda(a) \to \pi$ as $a \to a_*^{-}$.
\end{proof}

The above argument also gives the upper and lower bounds $1/\pi\leq a_*\leq 2$ for the critical 
value of $a$. 
It also shows that the eigenvalue $\lambda(a)$ satisfies $1/a<\lambda(a)\leq \pi+1/a$.
Indeed, the argument in step 6) gives that
$$
N((1/a,\infty); H(h_a))\geq1, 
$$
i.e. $\lambda(a)>1/a$. 
On the other hand, 
the norm of $H(h_a)$ satisfies
$$
\norm{H(h_a)}\leq \norm{H(\calL\nu_a)}+\norm{(\cdot,e^{-at/2})e^{-at/2}}=\pi+1/a,
$$
yielding the upper bound for $\lambda(a)$.


\begin{thebibliography}{1}

\bibitem{BE}
{\sc A. Erdelyi et al.,} \emph{Higher Transcendental Functions,} Bateman Manuscript Project, Vol. 2, McGraw--Hill, New York, 1953. 



\bibitem{BPSSV}
{\sc O. F. Brevig, K.-M. Perfekt, K. Seip, A. G. Siskakis, D. Vukoti\'c,}
\emph{The multiplicative Hilbert matrix,}
Adv. Math. \textbf{302} (2016), 410--432.


\bibitem{Helson}
{\sc H. Helson,}
\emph{Hankel forms,} 
Studia Math. \textbf{198} (2010), 79--83.



\bibitem{Howland}
{\sc J.~S.~Howland,}
\emph{Spectral theory of operators of Hankel type, II,}
Indiana University Mathematics Journal, 
\textbf{41} no. 2 (1992), 427--434.


\bibitem{Lebedev1}
{\sc N.~N.~Lebedev,}
\emph{Some singular integral equations connected with integral representations of mathematical physics. (Russian)}
Doklady Akad. Nauk SSSR (N.S.) \textbf{65} (1949), 621--624. 


\bibitem{Lebedev2}
{\sc N.~N.~Lebedev,}
\emph{The analogue of Parseval's theorem for a certain integral transform.} (Russian) 
Doklady Akad. Nauk SSSR (N.S.) \textbf{68} (1949), 653--656. 

\bibitem{Magnus}
{\sc W.~Magnus,}
\emph{On the spectrum of Hilbert's matrix,}
American Journal of Mathematics, \textbf{72}, no. 4 (1950), 699--704.

\bibitem{Power}
{\sc  S. R. Power,}
\emph{Hankel operators with discontinuous symbols,} 
Proc. Amer. Math. Soc. \textbf{65} (1977), 77--79.


\bibitem{PYa}
{\sc A.~Pushnitski, D.~Yafaev,} 
\emph{Spectral and scattering theory of self-adjoint Hankel operators with piecewise continuous symbols,}
J. Operator Theory \textbf{74}, no.2 (2015) 417--455. 

\bibitem{QQ}
{\sc H.~Queff\'elec, M.~Queff\'elec,}
\emph{Diophantine approximation and Dirichlet series.}
Hindustan Book Agency, New Delhi, 2013.  

\bibitem{Rosenblum}
{\sc M. Rosenblum,} 
\emph{On the Hilbert matrix, I, II,} 
Proc. Amer. Math. Soc., \textbf{9} (1958), 137--140,
581--585.

\bibitem{SaksmanSeip}
{\sc E.~Saksman, K.~Seip,} 
\emph{Some open questions in analysis for Dirichlet series.} 
Recent progress on operator theory and approximation in spaces of analytic functions, 
179--191, Contemp. Math., \textbf{679}, Amer. Math. Soc., Providence, RI, 2016. 


\bibitem{Titchmarsh}
{\sc E.~Titchmarsh,}
\emph{Introduction to the theory of Fourier integrals.}
Third edition. Chelsea Publishing Co., New York, 1986.

\bibitem{Yafaev1}
{\sc D.~R.~Yafaev, }
\emph{A Commutator Method for the Diagonalization
of Hankel Operators,}
Functional Analysis and Its Applications, \textbf{44} (2010), no. 4, 295--306.

\bibitem{Yafaev2}
{\sc D.~R.~Yafaev,}
\emph{Mathematical scattering theory. General theory.} 
Amer. Math. Soc., Providence, RI, 1992.




\end{thebibliography}
\end{document}